\documentclass[11pt]{article}
\usepackage{color,float,soul,graphicx}%
\usepackage{multirow,enumerate,enumitem}
\usepackage{url,cite,fullpage,fancyhdr,diffcoeff}
\usepackage{amsmath,amssymb,amsfonts,amsthm,mathrsfs,amscd}%
\usepackage{xcolor,textcomp,manyfoot}%
\usepackage{booktabs}%
\usepackage{algorithm,algorithmicx}%
\usepackage{algpseudocode,listings}%

\newcommand{\fenv}[1]%
{\ensuremath{\,\overrightarrow{\operatorname{env}}_{#1}}}
\newcommand{\benv}[1]%
{\ensuremath{\,\overleftarrow{\operatorname{env}}_{#1}}}

\renewcommand\theenumi{(\roman{enumi})}

\theoremstyle{thmstyleone}%
\theoremstyle{thmstyletwo}%
\theoremstyle{thmstylethree}%
\newtheorem{theorem}{Theorem}
\newtheorem{lemma}[theorem]{Lemma}

\newtheorem{proposition}[theorem]{Proposition}%
\newtheorem{definition}{Definition}%
\newtheorem{remark}{Remark}%
\newtheorem{alg}{Algorithm}[section]

\usepackage{cite}
\theoremstyle{definition}

\newcommand{\argmin}{{\rm arg}\!\min}

\newcommand{\la}{\langle}
\newcommand{\ra}{\rangle}

\newcommand{\Fix}{{\rm\textbf{Fix}\,}}

\newcommand{\nexto}{\kern -0.54em}

\newcommand{\dZ}{{\cal Z \kern -0.7em Z}}
\newcommand{\dC}{{\rm\hbox{C \kern-0.8em\raise0.2ex\hbox{\vrule height5.4pt width0.7pt}}}}
\newcommand{\dQ}{{\rm\hbox{Q \kern-0.85em\raise0.25ex\hbox{\vrule height5.4pt width0.7pt}}}}

\usepackage{epsfig}
\usepackage{latexsym}

\newcommand{\HH}{\mathcal{H}}

\newcommand{\KK}{\mathcal{K}}
\newcommand{\RR}{\mathbb{R}}

\newcommand{\ZZ}{\mathbb{Z}}

\date{}
\begin{document}
\title{Continuous and Discrete Systems for Quasi Variational Inequalities with Application to Game Theory}
\author{Oday Hazaimah
\footnote{E-mail: {\tt odayh982@yahoo.com}. https://orcid.org/0009-0000-8984-2500.  St. Louis, MO, USA}}
%

\maketitle

\begin{abstract}
A new class of projected dynamical systems of third order is investigated for quasi (parametric) variational inequalities in which the convex set in the classical variational inequality also depends upon the solution explicitly or implicitly. We study the stability of a continuous method of a gradient type. Some iterative implicit and explicit schemes are suggested as counterparts of the continuous case by inertial proximal methods. The convergence analysis of these proposed methods is established under sufficient mild conditions.  Moreover, some applications dealing with the generalized Nash equilibrium problems are presented.
\\

\noindent {\textbf{Keywords}:} Quasi-variational Inequalities; Projected Dynamical Systems; Discretization, Generalized Nash Equilibrium Games.
\medskip 

\noindent \textbf{\small Mathematics Subject Classification:}{ 34B15, 34B16, 65L05, 65L10, 65L11.}
\end{abstract}

\section{Introduction}
Variational inequalities is a general mathematical framework arising naturally in many theoretical and applied fields, such as finance, engineering, mechanics and operations research \cite{Blum, Cavazzuti, Dupuis, Dong-Zang-Nagurney, Goeleven, Giannessi, Hai, Kinder}. Variational inequalities involving a nonlinear term is called the mixed variational inequality while variational inequalities with a moving constraint set (i.e., the convex set in the variational inequality depends upon the solution) is called the parametric quasivariational inequality (QVI) which was first introduced by Bensoussan et al. \cite{Bensoussan-etal} and it will be our paradigm of study in this work. Note that if the involved set does not depend upon the solution then quasi variational inequality reduces to the Stampacchia variational inequality \cite{Stampacchia}. Variational inequality can be formulated in terms of dynamical systems to study the existence and stability of the solution \cite{Dupuis,Hazaimah-third-mvi,Noor-inequality}. 
It is well known that the solution of the variational inequality exists if either the constraint set is bounded or the corresponding mapping is strongly monotone. In recent years, the study of dynamical systems associated with variational inequalities provides qualitative insights for analyzing complex dynamics and optimizing systems by using methods of resolvent operators and projection operators over a set including; the proximal point algorithm, the gradient projection algorithm, among others (see, for instance, \cite{Alvarez,Antipin-extragrad, Bello-Hazaimah, Mija-prox,Mija-noor,Hai,Noor-inequality}). 
The convergence of the projection methods requires that the operator must be strongly monotone and Lipschitz continuous. 
Dynamical systems theory goes further than finite-dimensional variational inequalities since it allows for the study of the dynamics of equilibrium problems. The equilibrium points of a dynamical system form the solution set of the corresponding variational inequality problem. Unlike the existing extensive literature on variational inequalities, there is not much theory for parametric constraint situations.

Quasi-variational inequalities are known to be very useful for the modelling and analysis of many problems of equilibrium optimization and game theory. 
Solving QVI \eqref{QVI} by means of the optimization reformulation \eqref{QVI-optimization} provides global optimal solutions for general nonconvex objective functions. However, it is not easy to find global minimizers on the feasible set. Unlike in the particular case, the variational inequality, strict monotonicity is given on the objective function under which any equilibrium point of a dynamical system solves the variational inequality. Therefore it would be an interesting subject of future research to develop an optimization reformulation of the QVI that possesses the following property: any equilibrium point is essentially a global optimal solution. Several techniques using the notion of dynamical systems have been studied for approximating QVI, for instance, Khan et al. \cite{Khan-noor} considered second order dynamical system associated with quasi variational inequalities by applying some forward finite difference schemes. Mijajlovic and Jacimovic used proximal method in \cite{Mija-prox}, while the same authors considered continuous methods for solving quasi-variational inequalities in \cite{Mija-cont}. Mijajlovic et al. in \cite{Mija-noor} used gradient-type projection methods for quasi variational inequalities.
As it is known that the variational inequality can be extended to the Nash equilibrium game, generalized Nash equilibrium problems (GNEP) can be formulated in terms of QVIs. This formulation is referred to Bensoussan \cite{Bensoussan-etal} in which each player’s strategy set depends on the rival players strategies. The QVI has recently attracted growing attention in connection to game theory. Moreover, necessary and sufficient conditions for Nash equilibria of a game in normal strategic form cane be constructed in terms of the generated optimization problems.
\medskip

This paper aims at proposing a new continuous-time method of the third order and based on this design we also derive a variety of implicit and explicit discrete-time algorithms for solving parametric quasi-variational inequalities. Inspired by the applications of third-order ODEs which are used to describe and model the motion in electrical circuits \cite{Goeleven}, we adopt this idea in the manuscript. The suggested third-order projected dynamical system technique is very similar to the one used by Hazaimah \cite{Hazaimah-third-mvi} for mixed variational inequalities. It is worth noted that the resolvent dynamical system was proposed for mixed variational inequalities while in this note, the projected dynamical system is proposed for quasi variational inequalities. Moreover, the coefficients of the inertial and damping terms are constants in \cite{Hazaimah-third-mvi} while the coefficients in our case are time-dependent. To the best of our knowledge, this work is the first to use third-order dynamical systems to model QVIs by projection operators. 
Thus, our aim can be summarized as: (i) analysing a continuous gradient-type method with the most applicable form of the moving constraint set. (ii) using finite difference processes to identify the class of QVI by implicit and explicit discretizations for the associated dynamical system represented in terms of projection operators, (iii) discuss the global stability for solutions of the third-order dynamical system, and finally (iv) discuss some applications in the eyes of QVIs. 
\section{Mathematical Preliminaries}
Some mathematical foundations and significant definitions are presented in this section from monotone operators theory, dynamical systems theory, convex analysis and variational inequalities, see \cite{Kinder} for more details. Let $\HH$ be a real Hilbert space equipped with inner product $\la \cdot , \cdot \ra$ and induced norm $\|\cdot\|:=\sqrt{\la\cdot,\cdot\ra}$. Let $T:\HH\rightrightarrows\HH$ be a set-valued map. 
Let $\Fix(T):=\{x\in\HH :x\in T(x)\}$ be the set of all fixed points of the operator $T$.  
We are interested in designing dynamical systems models to derive discrete-time schemes for finding approximate solutions to the quasi variational inequality problem which can be formulated as finding \(x^*\in\KK(x^*)\), such that  
\begin{equation}\label{QVI}
\la T(x^*), \; x-x^*\ra\geq 0 \ , \ \forall \ x\in\KK(x^*)
\end{equation}
where $T$ is a general vector field and continunous map and \(\KK:\HH\rightrightarrows\HH\) is a dynamic constraint set-valued mapping such that \(\KK(x)\subseteq\HH\) is nonempty, closed and convex for all \(x\in\HH\). This parametric quasi variational inequality \eqref{QVI} was studied by Bensoussan et al. \cite{Bensoussan-etal}. If $\KK(x)\equiv\KK$, where \(\KK\) is a closed and convex set in $\HH$ then the parametric variational inequality \eqref{QVI} is equivalent to the classical variational inequality which was studied and considered by Stampacchia \cite{Stampacchia} as follows: find $x^*\in\HH$ such that 
\begin{equation}\label{VI}
\la T(x^*), \; x-x^*\ra\geq 0 \ , \ \forall \ x\in \KK.
\end{equation}
In many applied situations, the moving convex set \(\KK(x)\) has the form 
\begin{equation}\label{moving-convex-set}\KK(x)=k(x)+K_0.
\end{equation}
where $K_0\subseteq\HH$ is a fixed closed convex set and $k(x):\HH\to\HH$ is a continuous function satisfying the Lipschitz property i.e., \(\|k(x)-k(y)\|\leq l\|x-y\|, \ \text{for some positive} \ l>0.\)
Assuming \(x^*\in\Fix (T)\) is a fixed point of the operator \(T\) converges to the solution of the associated QVI \eqref{QVI}, then for a fixed \(x^*\in\HH\) the QVI is precisely a dynamic-constrained optimization problem 
\begin{equation}\label{QVI-optimization}
\min_{x\in\KK(x^*)}\la Tx^*,x-x^*\ra .
\end{equation}
If \(\KK^*(x)=\{x\in\HH:\la x,y\ra\geq 0, \ \forall y\in\KK(x)\}\) is a polar (dual set) of a convex-valued cone $\KK(x)$ in \(\HH\) then the inequality \eqref{VI} is equivalent to finding $x\in\KK$ such that $$x\in\KK(x), \ T(x)\in\KK^*(x) \ \ \text{and} \ \ \la Tx,x\ra =0,$$ 
which is called the quasi complementarity problem \cite{Kinder,Noor-Oettli}.
If the operator $T$ in \eqref{VI} is smooth, then the following well known result holds and can be viewed as a first order necessary optimality condition for minimizing smooth functions.
\begin{theorem} Let $\KK$ be a nonempty, convex and closed subset of $\HH$. Let $T$ be a smooth convex function. Then $x\in\KK$ is the minimum of the smooth convex $T(x)$ if and only if, $x\in\KK$ satisfies $$\la T'(x), y-x\ra\geq 0, \forall y \in\KK $$ where $T'$ is the Frechet derivative of $T$ at $x\in\KK$.
\end{theorem}
This theorem shows that the variational inequalities are analogous to the minimization of the convex differentiable functional subject to certain constraint which has led to study a more general framework of variational inequalities applied to nonconstrained and nonsmooth optimization problems. In the following, we state some useful definitions and properties for several kinds of monotone maps followed by well-known facts on projection operators and quasi variational inequalities. 
\begin{definition}\label{monotonedefi} The nonlinear operator $T:\HH\to\HH$, is said to be:
\begin{itemize}
\item [(i)] Monotone, if
$$\langle T(x)-T(y), x-y\rangle \geq 0,\ \ \forall x, y \in\HH.$$ 
\item [(ii)] Strictly monotone if the above inequality is strict for all $x\not= y$ in $\HH.$
\item [(iii)] Strongly monotone if there exists a modulus $\mu > 0$ such that
$$\langle T(x)-T(y), x-y\rangle\geq\mu\|x-y\|^2 , \ \ \forall x, y \in\HH.$$
\end{itemize}
\end{definition}
Notice that the implication $(iii)\implies (i)$ holds, whereas the converse need not be true generally, meaning that monotonicity is a weaker property than strongly monotonicity. 
\begin{definition}
The operator $T:\HH\to\HH$ is called Lipschitz continuous or $L$-Lipschitz if there exists some nonnegative $L\geq 0$, such that $$\|Tx-Ty\|\le L\|x-y\|\ ,\quad\forall x,y\in\HH.$$   
\end{definition}
\begin{lemma}\label{nec-suff-proj}
For a given $z\in\HH$. The necessary and sufficient characterizations of the projection are:
$$\Pi_{\KK(x)}(z)\in\KK(x),$$ if and only if 
$$\la\Pi_{\KK(x)}(z)-z, w-\Pi_{\KK(x)}(z)\ra\geq 0,\ \forall w\in\KK(x).$$
or equivalently, $\la z-\Pi_{\KK(x)}(z), w-\Pi_{\KK(x)}(z)\ra\leq 0,\ \forall w\in\KK(x),$
\end{lemma}
where \(\Pi_{\KK(x)}\) is called the implicit projection of \(\HH\) onto the closed convex-valued set \(\KK(x)\subset\HH.\) Note that the implicit projection \(\Pi_{\KK(x)}\) is nonexpansive (i.e., \(\|\Pi_{\KK(x)}(u)-\Pi_{\KK(x)}(v)\|\leq \|u-v\|\) ), and satisfy the condition 
\begin{equation}\label{implicit-projection}
\|\Pi_{\KK(x)}(u)-\Pi_{\KK(y)}(u)\|\leq \delta\|x-y\|, \ \ \forall x,y,u\in\HH,
\end{equation}
for some constant \(\delta >0.\) 
By applying Lemma \ref{nec-suff-proj}, one can introduce the fixed point formulation of parametric variational inequalities as follows.
\begin{proposition}[\cite{Jabeen-Noor}]\label{Jabeen-Noor}
Let \(\Pi_{\KK(x)}\) be the projection operator onto a closed convex set-valued \(\KK(x)\subset\HH\). Then \(x\in\KK(x)\) is a solution to the quasi variational inequality \eqref{QVI}, i.e., \[\la Tx,y-x\ra\leq 0, \ \ \forall y\in\KK(x)\] if and only if \ \(x=\Pi_{\KK(x)}(x-\lambda T(x)),\) \ for some \(\lambda>0.\)
\end{proposition}

A particular case if \(\KK(x)\equiv\KK\) then the implicit projection is exactly the Euclidean projection \(\Pi_{\KK(x)}=\Pi_{\KK}\) which is defined as 
\[\Pi_{\KK}(u):=\argmin_{x} \Big\{\displaystyle\frac{1}{2}\|x-u\|^2_2\Big\}.\]
Furthermore, for all $x,y\in \HH$ and all $z\in\KK$ the \(\Pi_\KK\) is firmly nonexpansive: 
\[\|\Pi_\KK(x)-\Pi_\KK(y)\|^2\leq \|x-y\|^2-\|(x-\Pi_\KK(x))-(y-\Pi_\KK(y))\|^2.\] 
Since $\KK(x)$ has the form \eqref{moving-convex-set}, we can write the implicit projection onto the dynamic convex set as
\begin{equation}\label{implicit-proj-onto-dynamic-set}
\Pi_{\KK(x)}(u)=\Pi_{k(x)+K_0}(u)=k(x)+\Pi_{K_0}(u-k(x)), \ \ \forall u\in\HH.
\end{equation}
Suppose that the first term in \eqref{moving-convex-set} is Lipschitz continuous with $l>0$, using 
\eqref{implicit-projection}, the Cauchy-Schwarz inequality and the nonexpansiveness of the Euclidean projection \(\Pi_K\), we then have 
\begin{equation}
\begin{split}
\|\Pi_{\KK(x)}(u)-\Pi_{\KK(y)}(u)\|=& \ \|k(x)-k(y)+\Pi_{K_0}(u-k(x))-\Pi_{K_0}(u-k(y))\| \\ \leq & \ \|k(x)-k(y)\|+\|\Pi_{K_0}(u-k(x))-\Pi_{K_0}(u-k(y))\| \\ \leq & \ 2\|k(x)-k(y)\|\leq 2l\|x-y\|, \ \ \ \forall x,y\in\HH.
\end{split}
\end{equation}
From Proposition \eqref{Jabeen-Noor}, it follows that \(x\in\KK(x)\) such that 
\[x=\Pi_{\KK(x)}(x-\lambda Tx)=\Pi_{k(x)+K_0}(x-\lambda Tx)=k(x)+\Pi_{K_0}(x-\lambda Tx-k(x)), \ \ \forall u\in\HH.\]
This implies that 
\[x-k(x)=\Pi_{K_0}(x-\lambda Tx-k(x)).\]
By Lemma \ref{nec-suff-proj}, this is equivalent to 
\begin{equation}\label{gvi}
\la Tx, g(y)-g(x)\ra\geq0, \ \ \forall y\in\HH,
\end{equation}
where \(g(y)=y\) and \(g(x)=x-k(x).\) Inequality \eqref{gvi} is called the general variational inequality, which was introduced by Noor \cite{Noor-GVI}, and is actually equivalent to the QVI \eqref{QVI}.
\begin{definition} 
The dynamical system is said to be globally exponentially stable if any trajectory $x(t)$ satisfies
\[\|x(t)-x^*\|\leq\rho\|x(t_0)-x^*\|exp(-\eta(t-t_0)), \ \ \ \forall t\geq t_0\]
where \(\rho,\eta >0\) are constants and do not depend on the initial point.
\end{definition}
If the dynamical system is stable at the equilibrium point $x^*$ in the Lyapunov sense then the dynamical system is globally asymptotically stable at that point. It is noted that globally exponentially stable means the system must be globally stable and converge fast. 
\medskip 

For deriving the convergence of our methods, the following theorem is needed:
\begin{theorem}[\cite{Vasiliev}]\label{Vasiliev} 
Let the operator $T:\HH\to\HH$ be a $\mu$-strongly monotone and $L$-Lipschitz continuous with $\mu , L >0.$ Then
\[\|T(x)-T(y)\|^2+ \mu L\|x-y\|^2\leq (L+\mu)\la T(x)-T(y), x-y\ra \ , \forall \ x, y\in\HH \]
holds.
\end{theorem}

\section{Main Results}
In this section, we introduce, derive and analyze continuous and discrete methods based on a third order dynamical system in the continuous case, and on the central finite difference and forward/backward difference methods in the discrete case besides their rates of convergences for quasi variational inequalities. 
Using the fixed point formulation, a new projected dynamical system of the third order associated with quasi variational inequalities \eqref{QVI} is proposed and some attached suitable discretizations forms are investigated. These continuous-time dynamical systems and their discrete-time counterparts suggest some inertial-type implicit and explicit proximal methods for solving quasi variational inequalities.
\subsection{Continuous gradient method}
The suggested projected dynamical system designed in a continuous-time form.
Consider the problem of finding a trajectory $x(t)\in\HH$ such that 
\begin{equation}\label{third-ord DS}
\left\{\begin{array}{lll} \alpha(t)\dddot{x}(t)+\beta(t)\ddot{x}(t)+\gamma(t)\dot{x}(t)+x(t) = \Pi_{\KK(x(t))}(x(t)-\lambda(t) T(x(t))), \\ x(t_0)=x_0, \\ \Dot{x}(t_0)=x_1, \\
\ddot{x}(t_0)=x_2, 
\end{array}\right.
\end{equation}
where \( x(t) \) is the state variable and the initial points \(x_0,x_1,x_2\in\HH\). 
The differential system \eqref{third-ord DS} recovers several existing dynamics-type approaches and projection-based algorithms for solving several classes of variational inequalities. This model \eqref{third-ord DS} is quite similar to the dynamical system used by Hazaimah \cite{Hazaimah-third-mvi} for mixed variational inequalities, the difference between the two approaches is by taking the coeffecients of the first three terms on the left hand side of \eqref{third-ord DS} to be time scaling. Particular cases of the general system \eqref{third-ord DS} are discussed next. If $\alpha(t)\equiv 0, \beta(t)\equiv\beta, \gamma(t)\equiv 1$, then \eqref{third-ord DS} is equivalent to the continuous second-order dynamical system introduced by Antipin et al. \cite{Antipin-cont}, while the same technique was examined in the discrete case in \cite{Noor-DS for QVI} with implicit iterative methods. Extra gradient method for solving quasi variational inequalities is introduced in \cite{Antipin-extragrad}. Some second-order iterative versions are studied in \cite{Antipin-iterative}. 
If $\alpha(t)\equiv0\equiv\beta(t)\equiv\gamma(t)$ and \(\KK(x)=\KK\), then the system \eqref{third-ord DS} is reduced to the classical gradient projection for smooth constrained optimization problems and projection-like methods for solving variational inequalities.
\medskip 

The projected dynamical system \eqref{third-ord DS} can be rewritten, utilizing the mapping form \eqref{implicit-proj-onto-dynamic-set}, as    
\begin{equation}\label{third-ord-set-valued}
\alpha(t)\dddot{x}(t)+\beta(t)\ddot{x}(t)+\dot{x}(t)+x(t)=k(x(t))+\Pi_{K_0}\Big(x(t)-\lambda(t) T(x(t))-k(x(t))\Big).
\end{equation}
The following theorem discusses the exponentially stability and shows the convergence rate for the proposed continuous method.
\begin{theorem}
Assume the following:
\begin{enumerate}
\item $\KK(x):\HH\rightrightarrows\HH$ is a convex set-valued mapping satisfies \eqref{moving-convex-set} where $k(x):\HH\to\HH$ is $l$-Lipschitz mapping.
\item $T$ is $\mu$-strongly monotone and Lipschitzian with constant $L>0$.
\item $\alpha(1-l)\geq0$, $(1-l)(\alpha+\beta)\geq 3\alpha$.
\item $l<\min\{1, \ 1-2\beta, \ 2\lambda(t)\big[\mu-lL\big]-\mu^2\lambda^2(t)\}$.
\item $\lambda (t)\in C^1[0,\infty), \ \alpha(t)\equiv\alpha >0, \ \beta(t)\equiv\beta >0,\ \gamma(t)\equiv 1.$ 
\end{enumerate}
Then 
\begin{equation}
\begin{split}
\|x(t)-x^*\|^2\leq [e^{\frac{\beta}{\alpha}t}]^{-1}\|x(0)-x^*\|^2+\alpha^{-1}K_2\ e^{-\frac{2\beta}{\alpha}t} , \ \ \forall t\geq 0
\end{split}
\end{equation}
\end{theorem}

\begin{proof}
Using Proposition \ref{Jabeen-Noor}, the projected dynamical system given by \eqref{third-ord-set-valued} is equivalent to the variational inequality setting 
\begin{equation*}
\begin{split}
\la x(t)-\lambda(t)T(x(t))&-k(x(t)) -\alpha\dddot{x}(t)-\beta\ddot{x}(t)-\dot{x}(t)-x(t)+k(x(t)), \\ & y-\alpha\dddot{x}(t)-\beta\ddot{x}(t)-\dot{x}(t)-x(t)+k(x(t))\ra\leq 0,
\end{split}
\end{equation*}
for all \(y\in K_0\). Rearranging the above inequality using simple algebraic manipulations
\begin{equation}\label{PDS==VI}
\begin{split}
\la\lambda(t)T(x(t))&+\alpha\dddot{x}(t)+\beta\ddot{x}(t)+\dot{x}(t),y-\alpha\dddot{x}(t)-\beta\ddot{x}(t)-\dot{x}(t)-x(t)+k(x(t))\ra\geq 0,
\end{split}
\end{equation}
for all \(y\in K_0\). Set $y=x^*-k(x^*)\in K_0$ in \eqref{PDS==VI} and $x=\alpha(t)\dddot{x}(t)+\beta(t)\ddot{x}(t)+\dot{x}(t)+x(t)+k(x^*)-k(x(t))\in\KK(x^*)$ in \eqref{QVI}. 
Multiply \eqref{QVI} by $\lambda(t)>0$ and combining the resulting inequalities, we obtain respectively
\begin{equation*}
\begin{split}
\la\lambda(t)T(x(t))&+\alpha\dddot{x}(t)+\beta\ddot{x}(t)+\dot{x}(t),-\alpha\dddot{x}(t)-\beta\ddot{x}(t)-\dot{x}(t)-x(t)+x^*+k(x(t))-k(x^*)\ra\geq 0.
\end{split}
\end{equation*}
and
\begin{equation*}
\begin{split}
\lambda(t)\la T(x^*),\alpha\dddot{x}(t)+\beta\ddot{x}(t)+\dot{x}(t)+x(t)-x^*+k(x^*)-k(x(t))\ra\geq 0.
\end{split}
\end{equation*}
Adding together, we have 
\begin{equation*}
\begin{split}
&\la\lambda(t)T(x(t)),-\alpha\dddot{x}(t)-\beta\ddot{x}(t)-\dot{x}(t)-x(t)+x^*+k(x(t))-k(x^*)\ra\\&-\la\alpha\dddot{x}(t)+\beta\ddot{x}(t)+\dot{x}(t),\alpha\dddot{x}(t)+\beta\ddot{x}(t)+\dot{x}(t)+x(t)-x^*+k(x^*)-k(x(t))\ra\\&+\lambda(t)\la T(x^*),\alpha\dddot{x}(t)+\beta\ddot{x}(t)+\dot{x}(t)+x(t)-x^*+k(x^*)-k(x(t))\ra\geq 0.
\end{split}
\end{equation*}
Thus
\begin{equation}\label{before-inequality-ab}
\begin{split}
\la\alpha&\dddot{x}(t)+\beta\ddot{x}(t)+\dot{x}(t),\alpha\dddot{x}(t)+\beta\ddot{x}(t)+\dot{x}(t)+x(t)-x^*+k(x^*)-k(x(t))\ra \\&\leq\lambda(t)\la T(x(t))-T(x^*),-\alpha\dddot{x}(t)-\beta\ddot{x}(t)-\dot{x}(t)-x(t)+x^*+k(x(t))-k(x^*)\ra
\end{split}
\end{equation}
Next we rewrite \eqref{before-inequality-ab}, by applying the inequality \(ab\leq\displaystyle\frac{a^2}{2}+\displaystyle\frac{b^2}{2}\)
, as 
\begin{equation}\label{after-inequality-ab}
\begin{split}
\|\alpha\dddot{x}(t)&+\beta\ddot{x}(t)+\dot{x}(t)\|^2+\la\alpha\dddot{x}(t)+\beta\ddot{x}(t)+\dot{x}(t),x(t)-x^*\ra+\la\alpha\dddot{x}(t)+\beta\ddot{x}(t)+\dot{x}(t),\\&k(x(t))-k(x^*)\ra\leq\frac{\lambda^2(t)}{2}\|T(x(t))-T(x^*)\|^2+\frac{1}{2}\|\alpha\dddot{x}(t)-\beta\ddot{x}(t)-\dot{x}(t)\|^2 \\&+\lambda(t)\la T(x(t))-T(x^*),x^*-x(t)\ra+\lambda(t)\la T(x(t))-T(x^*),k(x(t))-k(x^*)\ra
\end{split}
\end{equation}
From the assumption $T$ is $L$-Lipschitz and $\mu$-strongly monotone. Applying Theorem \ref{Vasiliev} and since $k(x)$ is $l$-Lipschitz, the inequality \eqref{after-inequality-ab} implies that 
\begin{equation*}
\begin{split}
\|\alpha\dddot{x}(t)&+\beta\ddot{x}(t)+\dot{x}(t)\|^2+2\la\alpha\dddot{x}(t)+\beta\ddot{x}(t)+\dot{x}(t),x(t)-x^*\ra 
-l\|\alpha\dddot{x}(t)+\beta\ddot{x}(t)+\dot{x}(t)\|^2 \\& -l\|x(t)-x^*\|^2+\lambda(t)\big[2-(L+\mu)\lambda(t)\big]\la T(x(t))-T(x^*),x(t)-x^*\ra \\& +\lambda^2(t)L\mu\|x(t)-x^*\|^2-2\lambda(t)lL\|x(t)-x^*\|^2\leq 0.
\end{split}
\end{equation*}
Since $T$ is $\mu$-strongly monotone then by Definition \ref{monotonedefi} (iii), the latter inequality reduced to 
\begin{equation*}
\begin{split}
(1&-l)\|\alpha\dddot{x}(t)+\beta\ddot{x}(t)+\dot{x}(t)\|^2+2\la\alpha\dddot{x}(t)+\beta\ddot{x}(t)+\dot{x}(t),x(t)-x^*\ra \\& +\mu\lambda(t)\big[2-(L+\mu)\lambda(t)\big]\|x(t)-x^*\|^2 +\Big(\lambda^2(t)L\mu-2\lambda(t)lL-l\Big)\|x(t)-x^*\|^2\leq 0.
\end{split}
\end{equation*}
Algebraically rearranging terms, we obtain 
\begin{equation}\label{simplyineq}
\begin{split}
(1&-l)\|\alpha\dddot{x}(t)+\beta\ddot{x}(t)+\dot{x}(t)\|^2+2\la\alpha\dddot{x}(t)+\beta\ddot{x}(t)+\dot{x}(t),x(t)-x^*\ra+A(t)\|x(t)-x^*\|^2\leq 0,
\end{split}
\end{equation}
where \(A(t)=2\lambda(t)\big[\mu-lL\big]-\mu^2\lambda^2(t)-l.\)

Next, we employ the following helpful relations to simplify the previous inequality:
\begin{subequations}
\begin{align}
\la\dot{x},\dot{x}\ra &= \|\dot{x}\|^2 \\ 
\la\dot{x},x(t)-x^*\ra &=\frac{1}{2}\frac{d}{dt}\|x(t)-x^*\|^2 \\ 
\la\ddot{x},x(t)-x^*\ra &=\frac{1}{2}\frac{d^2}{dt^2}\|x(t)-x^*\|^2-\|\dot{x}\|^2 \\ 
\la\dddot{x},x(t)-x^*\ra &=\frac{1}{2}\frac{d^3}{dt^3}\|x(t)-x^*\|^2-3\la\ddot{x},\dot{x}\ra \\ 
2\la\dddot{x},\dot{x}\ra &=\frac{d}{dt} \|\ddot{x}\|^2+\frac{d}{dt}\|\dot{x}\|^2-2\|\ddot{x}\|^2-\frac{d}{dt}\|\ddot{x}-\dot{x}\|^2
\end{align}
\end{subequations}
Hence, inequality \eqref{simplyineq} can be rewritten as 
\begin{equation}\label{moresimplify}
\begin{split}
(1&-l)\bigg[\alpha^2\|\dddot{x}(t)\|^2+\beta^2\|\ddot{x}(t)\|^2+\|\dot{x}(t)\|^2+2\alpha\beta\la\dddot{x},\ddot{x}\ra+2\alpha\la\dddot{x},\dot{x}\ra+2\beta\la\ddot{x},\dot{x}\ra\bigg]\\&+\alpha\frac{d^3}{dt^3}\|x(t)-x^*\|^2 -6\alpha\la\ddot{x},\dot{x}\ra +\beta\frac{d^2}{dt^2}\|x(t)-x^*\|^2-2\beta\|\dot{x}\|^2+\frac{d}{dt}\|x(t)-x^*\|^2\\&+A(t)\|x(t)-x^*\|^2\leq 0,
\end{split}
\end{equation}
Using Rewriting \eqref{moresimplify}, we have 
\begin{equation}\label{third-derivative}
\begin{split}
\alpha&\frac{d^3}{dt^3}\|x(t)-x^*\|^2+\beta\frac{d^2}{dt^2}\|x(t)-x^*\|^2+\frac{d}{dt}\|x(t)-x^*\|^2+A(t)\|x(t)-x^*\|^2 \\& -\alpha(1-l)\frac{d}{dt}\|\ddot{x}-\dot{x}\|^2+\alpha(1-l)(1+\beta)\frac{d}{dt}\|\ddot{x}\|^2+\bigg((1-l)(\alpha+\beta)-3\alpha\bigg)\frac{d}{dt}\|\dot{x}\|^2\\& +(1-l)\alpha^2\|\dddot{x}\|^2+(1-l)\beta^2\|\ddot{x}\|^2+(1-l-2\beta)\|\dot{x}\|^2\leq 0,
\end{split}
\end{equation}
Multiply \eqref{third-derivative} by \(H(t)=\exp{\int_0^tds}\), and integrating over the interval \([0,t]\), we obtain 
\begin{equation}\label{third-derivative-reduced}
\begin{split}
\alpha&\frac{d^2}{dt^2}\|x(t)-x^*\|^2+\beta\frac{d}{dt}\|x(t)-x^*\|^2+\|x(t)-x^*\|^2+\int_0^tA(s)H(s)\|x(s)-x^*\|^2ds \\& -\alpha(1-l)\|\ddot{x}(t)-\dot{x}(t)\|^2+\alpha(1-l)(1+\beta)\|\ddot{x}(t)\|^2+\bigg((1-l)(\alpha+\beta)-3\alpha\bigg)\|\dot{x}(t)\|^2\\& +(1-l)\alpha^2\int_0^tH(s)\|\dddot{x}(s)\|^2ds+(1-l)\beta^2\int_0^tH(s)\|\ddot{x}(s)\|^2ds\\&+(1-l-2\beta)\int_0^tH(s)\|\dot{x}(s)\|^2ds\leq K_0,
\end{split}
\end{equation}
where 
\[K_0=\|x(0)-x^*\|^2-\alpha(1-l)\|\ddot{x}(0)-\dot{x}(0)\|^2+\alpha(1-l)(1+\beta)\|\ddot{x}(0)\|^2+\Big((1-l)(\alpha+\beta)-3\alpha\Big)\|\dot{x}(0)\|^2.\] 
All definite integrals on the left hand side of \eqref{third-derivative-reduced} are non-negative, which means that 
\begin{equation}\label{all-integrals-positive}
\begin{split}
\alpha&\frac{d^2}{dt^2}\|x(t)-x^*\|^2+\beta\frac{d}{dt}\|x(t)-x^*\|^2+\|x(t)-x^*\|^2 -\alpha(1-l) \|\ddot{x}(t)-\dot{x}(t)\|^2 \\& +\alpha(1-l)(1+\beta)\|\ddot{x}(t)\|^2+\bigg((1-l)(\alpha+\beta)-3\alpha\bigg)\|\dot{x}(t)\|^2 \leq K_0, \ \ \forall t\geq 0
\end{split}
\end{equation}
Similarly, we integrate the inequality \eqref{all-integrals-positive} over the interval $[0,t]$, then \eqref{all-integrals-positive} simplified to 
\begin{equation}\label{first-derivative-left}
\begin{split}
\alpha&\frac{d}{dt}\|x(t)-x^*\|^2+\beta\|x(t)-x^*\|^2+\int_0^t\|x(s)-x^*\|^2ds-\alpha(1-l)\int_0^t \|\ddot{x}(s)-\dot{x}(s)\|^2ds \\& +\alpha(1-l)(1+\beta)\int_0^t\|\ddot{x}(s)\|^2ds+\bigg((1-l)(\alpha+\beta)-3\alpha\bigg)\int_0^t\|\dot{x}(s)\|^2ds \leq K_0+K_1, \ \ \forall t\geq 0
\end{split}
\end{equation}
where \(K_1=\beta\|x(0)-x^*\|^2\). In the same manner, we observe that all integrals in the left hand side are non-negative, and therefore \eqref{first-derivative-left} becomes 
\begin{equation}\label{K0K1=K2}
\begin{split}
\alpha\frac{d}{dt}\|x(t)-x^*\|^2+\beta\|x(t)-x^*\|^2\leq K_0+K_1=K_2, \ \ \forall t\geq 0
\end{split}
\end{equation}
Multiply \eqref{K0K1=K2} by $\alpha^{-1}$, we obtain the linear differential inequality 
\begin{equation}\label{lin-diff-ineq}
\begin{split}
\frac{d}{dt}\|x(t)-x^*\|^2+\frac{\beta}{\alpha}\|x(t)-x^*\|^2\leq\alpha^{-1}K_2, \ \ \forall t\geq 0
\end{split}
\end{equation}
Multiply \eqref{lin-diff-ineq} by 
\(e^{\int_0^t\frac{\beta}{\alpha}ds}\), so it can be rewritten, using the idea of the integrating factor, in the form 
\begin{equation}\label{integrating-factor}
\begin{split}
\frac{d}{dt}\Big[\|x(t)-x^*\|^2e^{\int_0^t\frac{\beta}{\alpha}ds}\Big] \leq\alpha^{-1}K_2 \ e^{-\frac{\beta}{\alpha}t}, \ \ \forall t\geq 0
\end{split}
\end{equation}
Integrating \eqref{integrating-factor} over the interval \([0,t]\), we have 
\begin{equation}\label{}
\begin{split}
\|x(t)-x^*\|^2e^{\frac{\beta}{\alpha}t} \leq 
\|x(0)-x^*\|^2+\alpha^{-1}K_2 \ e^{-\frac{\beta}{\alpha}t}, \ \ \forall t\geq 0
\end{split}
\end{equation}
Multiplying by \([e^{\frac{\beta}{\alpha}t}]^{-1}\), it follows that 
\begin{equation}\label{conv-rate}
\begin{split}
\|x(t)-x^*\|^2\leq [e^{\frac{\beta}{\alpha}t}]^{-1}\|x(0)-x^*\|^2+\alpha^{-1}K_2\ e^{-\frac{2\beta}{\alpha}t} , \ \ \forall t\geq 0
\end{split}
\end{equation}
Thus, we have proved the result with an estimation rate of convergence of the continuous method \eqref{third-ord DS} given by inequality \eqref{conv-rate}.
\end{proof}
\begin{remark}
If we let the coefficient parameters of the inertial and damping terms \(\ddot{x}, \dot{x}\) respectively, change over time, namely \(\alpha(t)\not\equiv\alpha, \beta(t)\not\equiv\beta\) then the convergence of the continuous method \eqref{third-ord DS} can be improved under some mild general conditions such that 
\[\alpha(t),\beta(t)\in C^2[0,\infty), \ \ t\geq 0, \ \ \lim_{t\to\infty}\alpha(t)>0, \ \lim_{t\to\infty}\beta(t)>0.\]
\end{remark}

\subsection{Discrete methods}
The idea of iterative methods mostly is based on discretizing the space derivatives by using certain discretizations methods and the goal behind boosting iterative schemes is to accelerate the rate of convergence. In this paper we use the central finite difference, backward difference and forward difference schemes to propose explicit and implicit forms which enable us to obtain the discretized counterpart of \eqref{third-ord DS} as a projected equation. Thus the dynamical system \eqref{third-ord DS} may be discretized as: 
\begin{equation}\label{discrete}
\begin{split}
\alpha\displaystyle\frac{x_{n+2}-2x_{n+1}+2x_{n-1}-x_{n-2}}{2h^3} & +\beta\frac{x_{n+1}-2x_n+x_{n-1}}{h^2}+\gamma\frac{x_n-x_{n-1}}{h}+x_{n+2}\\ &=\Pi_{\KK(x_n)}(x_n-\lambda T(x_{n+2}))
\end{split}
\end{equation}
where $h$ is the step size for the iterative process. This discrete scheme \eqref{discrete} suggests a new \textit{implicit} iterative method for solving quasi variational inequalities \eqref{QVI} by the third order central difference formula.
\begin{alg}\label{Algorithm1}
For any $x_0,x_1,x_2\in\HH$, and for any nonnegative integer $n\in\ZZ_+$, compute the update rule $x_{n+2}$ by the iterative process 
\begin{equation}\label{algorithm}
\begin{split}
x_{n+2} & =\Pi_{\KK(x_n)}\bigg[ x_n-\lambda T(x_{n+2}) \\ & -\frac{\alpha x_{n+2}-2(\alpha-\beta h)x_{n+1}-2(2\beta h-\gamma h^2)x_n+2(\alpha+\beta h-\gamma h^2)x_{n-1}-\alpha x_{n-2}}{2h^3}\bigg]
\end{split}
\end{equation}
\end{alg}
This algorithm is inertial proximal-type method for solving \eqref{QVI}. Using Lemma \ref{nec-suff-proj}, Algorithm \ref{Algorithm1} can be rewritten in the variational equivalent formulation: 
\begin{alg}
For any $x_0,x_1,x_2\in\HH$, and for any nonnegative integer $n\in\ZZ_+$, compute $x_{n+2}$ by the iterative process 
\begin{equation}\label{MVI-alg}
\begin{split}
\big\la\lambda T(x_{n+2})&+\frac{\alpha x_{n+2}-2(\alpha-\beta h)x_{n+1}-2(2\beta h-\gamma h^2)x_n+2(\alpha+\beta h-\gamma h^2)x_{n-1}-\alpha x_{n-2}}{2h^3}, \\ & y-x_{n+2}\big\ra\geq 0, \ \ \ \forall y\in\KK(x)
\end{split}
\end{equation}
\end{alg}
Using different discretization and taking $\alpha=1=\beta=\gamma$, Algorithm \ref{Algorithm1} can be reduced to the following iterative: 
\begin{equation*}
\begin{split}
\frac{x_{n+2}-2x_{n+1}+2x_{n-1}-x_{n-2}}{2h^3} & +\frac{x_{n+1}-2x_n+x_{n-1}}{h^2} + \frac{x_n-x_{n-1}}{h}+x_{n+2}\\ &=\Pi_{\KK(x_n)}(x_n-\lambda T(x_{n}))
\end{split}
\end{equation*}
which yields to the following recurrence formula
\begin{equation}\label{explicit}
x_{n+2} =\frac{\hat{h}}{1+\hat{h}}\Pi_{\KK(x_n)}\bigg[ (1-\frac{1}{h}+\frac{2}{h^2})x_n-\lambda T(x_{n}) -\frac{(2h-2)x_{n+1}+(2+2h-2h^2)x_{n-1}-x_{n-2}}{2h^3}\bigg]
\end{equation}
where $\hat{h}=2h^3$. This is called an inertial \textit{explicit} proximal method for solving quasi variational inequalities \eqref{QVI}. 
Taking $\alpha=1=\beta=\gamma$, and \(h=1\) then Algorithm \ref{Algorithm1} can be reduced to the explicit iterative formula:
\begin{equation}\label{short-iterative}
3x_{n+2}-2x_n+x_{n-2}=2\ \Pi_{\KK(x_n)}(x_n-\lambda T(x_{n})).
\end{equation}
Following the same fashion with slightly exploring the forward/backward iterates we can suggest several explicit and implicit recursive methods for finding approximate solutions of parametric quasi variational inequalities \eqref{QVI}. Hence, by using the central finite difference and this time with forward difference scheme rather than backward scheme as in \eqref{discrete}, which allows us to propose a new iterative approach  
\begin{equation}\label{last-implicit}
\begin{split}
\alpha\displaystyle\frac{x_{n+2}-2x_{n+1}+2x_{n-1}-x_{n-2}}{2h^3} & +\beta \frac{x_{n+1}-2x_n+x_{n-1}}{h^2} +\gamma \frac{x_{n+1}-x_n}{h}+x_{n+2}\\ &=\Pi_{\KK(x_n)}(x_n-\lambda T(x_{n+1}))
\end{split}
\end{equation}
which can be, equivalently, derived as the following inertial \textit{implicit} proximal method:
\begin{alg}\label{Algorithm2}
For $x_0,x_1,x_2\in\HH$, and for any nonnegative integer $n\in\ZZ_+$, compute the update step $x_{n+2}$ by  
\begin{equation}
\begin{split}
x_{n+2} & =\Pi_{\KK(x_n)}\bigg[ x_n-\lambda T(x_{n+1}) \\ & -\frac{\alpha x_{n+2}-2(\alpha-\beta h-\gamma h^2)x_{n+1}-2(2\beta h+\gamma h^2)x_n +2(\alpha+\beta h)x_{n-1}-\alpha x_{n-2}}{2h^3}\bigg].
\end{split}
\end{equation}
\end{alg}
For $\alpha=1=\beta=\gamma$ and \(h=1\) then Algorithm \ref{Algorithm2} can be reduced to the explicit version: 
\begin{equation*}
3x_{n+2}+2x_{n+1}-6x_n+4x_{n-1}-x_{n-2}=2\ \Pi_{\KK(x_n)}(x_n-\lambda T(x_{n+1})).
\end{equation*}
On a different perspective, it is known that adding an inertial term into discrete-time algorithms will speed up and cause a significant change in the convergence rate using extrapolating factor \(\Theta_n(x_n-x_{n-1})\) for several classes of smooth and strongly monotone mappings. In the light of this concept, we introduce the following two-step inertial iterative algorithms:
\begin{algorithm}
\begin{alg}
Given \(x_0,x_1,x_2\in\HH\) and \(n\in\ZZ_+\), compute \(x_{n+2}\) by the iterative steps:
\begin{equation*}
\begin{split}
 z_n=& \ x_n+\Theta_n(x_n-x_{n-1}) \\ 
x_{n+2}=& \ \frac{2}{3}\ \Pi_{\KK(x_n)}(3x_n-\lambda Tz_n-x_{n-1})
\end{split}
\end{equation*}
where \(0\leq\Theta_n\leq1\).
\end{alg}
\end{algorithm}

Similarly, given \(x_0,x_1,x_2\in\HH\) and \(0\leq\Theta_n\leq1\) for \(n\in\ZZ_+\), in some cases letting the constraint set depend on the inertial equation would give new algorithms. Thus, we can compute \(x_{n+2}\) by two new inertial iterative methods for quasi variational inequalities, respectively:
\begin{equation*}
\begin{split}
 z_n=& \ x_n+\Theta_n(x_n-x_{n-1}) \\ 
x_{n+2}=& \ \frac{2}{3} \ \Pi_{\KK(x_n)}(3z_n-\lambda Tz_n-x_{n-1}),
\end{split}
\end{equation*}
and
\begin{equation*}
\begin{split}
 z_n=& \ x_n+\Theta_n(x_n-x_{n-1}) \\ 
x_{n+2}=& \ \frac{2}{3} \ \Pi_{\KK(z_n)}(z_n-\lambda Tz_n-x_{n-1}).
\end{split}
\end{equation*}

Before wrapping up this section and going to the convergence analysis it is worth noting that by applying suitable discretizations based on changing the update rule explicitly or implicitly, one can establish and design a variety of inertial projection proximal-type methods for solving parametric quasi variational inequalities \eqref{QVI}. Convergence analyses for Algorithm \ref{Algorithm1} of the third-order projected dynamical system \eqref{third-ord DS} are derived in the remaining part of this work. 

\subsection{Convergence of a discrete system}
In this section, we derive the convergence of a solution to the implicit iterative scheme \eqref{algorithm} and its equivalent variational form \eqref{MVI-alg} given by Algorithm \eqref{Algorithm1}. However, other implicit \eqref{last-implicit} and explicit \eqref{explicit} proposed methods have a very similar arguments and follow the same guidlines except that there are some minor differences which is due to the values of the scalars formatting of $\alpha,\beta,\gamma$, and also due to the existing diverse discretization schemes. In proving that the approximate solution converges to a unique accumulation point, we need the following assumption:
\medskip 

\textbf{Assumption 1.} Suppose that \(x_n\to x\) as \(n\to\infty\), then for any \(y\in\KK(x)\) there exists a sequence \(\{y_n\}\) sucht hat \(y_n\in\KK(x_n)\) and \(y_n\to y\). For all sequences \(\{x_n\}\) and \(\{y_n\}\) such that 
\(y_n\in\KK(x_n)\), then \(y\in\KK(x)\).

\begin{theorem}\label{nonincreasing}
Let $x\in\KK(x)$ be the solution of the quasi variational inequality \eqref{QVI} and $x_{n+2}$ be the approximate solution using the inertial proximal method in \eqref{MVI-alg}. If $T$ is monotone, then 
\begin{equation}\label{theorem-result}
\begin{split}
(\alpha-\beta h+\gamma h^2)\|x &-x_{n+2}\|^2 \leq\ \alpha\|x-2x_{n+1}+2x_{n-1}-x_{n-2}\|^2 \\ &-\alpha \|x_{n+2}-2x_{n+1}+2x_{n-1}-x_{n-2}\|^2+ \beta h\|x_{n+1}-2x_{n}+x_{n-1}\|^2 \\ & +\gamma h^2\|x_n-x_{n-1}+x-x_{n+2}\|^2-\gamma h^2\|x_{n}-x_{n-1}\|^2.
\end{split}
\end{equation}
\end{theorem}
\begin{proof}
See \cite{Hazaimah-third-mvi}. 
\end{proof}

\begin{theorem}
Let $x\in\KK(x)$ be the solution of \eqref{QVI}. Let $x_{n+2}$ be the approximate solution of Algorithm \ref{Algorithm1}, Suppose that the operator $T$ is monotone and Assumption 1 satisfies, then the generated sequence from \eqref{MVI-alg} converges to the solution $x$ of the parametric quasi variational inequality \eqref{QVI}
, i.e., 
$\displaystyle\lim_{n\to\infty}x_{n+2}=x.$
\end{theorem}
\begin{proof}
See \cite{Hazaimah-third-mvi}.
\end{proof}

\section{Applications}
We investigate two scenarios. The first one is the obstacle problem which is a class of free boundary problems that observe the dynamic behavior of a state variable described by a differential equation and studies the equilibrium states over an obstacle with fixed boundary conditions arise in financial mathematics and optimal control. The second one is the generalized Nash equilibrium problem which is an extension of the classical Nash equilibrium problem, in which each player’s strategy set depends on the rival player’s strategies.
\subsection{The obstacle boundary value}
Consider the second-order implicit obstacle boundary value problem, which have been discussed in
Noor \cite{Noor-some-qvi} as finding \(x\) such that 
\[ \left\{ \begin{array}{ll}
-\ddot{x}(t)\geq f(t) & \text{on}\; \Omega=[a,b]\\
x(t)\geq M(x(t)) & \text{on}\; \Omega=[a,b]\\
\big(-\ddot{x}(t)-f(t)\big)\big(x(t)-M(x(t))\big)=0 &  \text{on}\;  \Omega=[a,b]\\ 
x(a)=0=x(b),
\end{array}\right.\]
where \(f(t)\) is a continuous function and \(M(x(t))\) is the following cost function
\[M(x(t))=k+\inf_{i}x^i , \ \ \ \text{where} \ k \geq 0.\]
To see the connection between the obstacle problem and the quasi variational inequalities, we define the constraint closed convex-valued set 
\[\KK(x)=\{y:y\in\HH_0^1(\Omega), \ \ y\geq M(x)\}\]
where \(\HH_0^1(\Omega)\) is a Sobolev space. Introduce the energy functional corresponding to the obstacle problem as:
\[I[y]=\la Tx,y\ra-2\la f,y\ra, \ \ \forall y\in\KK(x)\]
where 
\[\la Tx,y\ra =\int_a^b \bigg(\frac{dy}{dt}\bigg)^2dt, \ \ \text{and} \ \ \la f,y\ra=\int_a^bf(t)y\ dt.\]

It is clear that the operator $T$ defined above is linear, symmetric and positive. Using the
technique of Noor \cite{newtrends-gvi} one can show that the minimum of the
functional $I$ associated with the problem (2) on the closed convex-valued
set $\KK(x)$ can be characterized by the inequality 
\[\la Tx,y-x\ra\geq\la f,y-x\ra, \ \ \forall y\in\KK(x).\]

which is exactly the quasi variational inequality \eqref{QVI}.

\subsection{Generalized Nash equilibrium problems}
The main concept in game theory is the Nash equilibrium. A Nash equilibrium is a set of strategies (constraints) assigned to each member (player) of the game. In other words, it is one such that no constraints across players are allowed. 
It is known that variational inequalities or variational equilibria can be extended to the Nash equilibria game. In the same fashion, generalized Nash equilibrium problems can be formulated in terms of QVIs due to Bensoussan \cite{Bensoussan-etal}, where the author dealt with infinite-dimensional strategy sets in which not only each player’s payoff function but also their strategy set depend on the other players strategies. Necessary and sufficient conditions for Nash equilibria of a game in normal strategic form cane be constructed in terms of the generated optimization problems. 

\textbf{Noncooperative Games in normal strategic form:}
Roughly speaking, when the feasible set of the game is actually the full Cartesian product of the individual strategy sets then the composed game is called a noncooperative game. In other words, players can only impact the cost functions of the rival players but not their feasible sets.
Consider a finite set \(I=\{1, ...,n\}\) of players such that each player \(i\in I\) has a set of strategies (also called actions or constraints) denoted by \(S_i\), and this set is a compact convex subset of a Hilbert space \(\HH_i\) defined by the set-valued mapping \(S_i:\HH\setminus\HH_i\rightrightarrows\HH_i\) where \(\HH=\prod_{i\in I}\HH_i\) is the ambient space. Each player $i$ controls their decision variable \(x^i\in\HH_i\) such that the vector $x=(x^1,..., x^n)\in\HH$ describes the decision vector of all players. We often use the notation $x=(x^i, x^{–i})$, where $x^{–i}=(x^j)^{j\in I\setminus\{i\}}=(x^1, x^2,..., x^{i–1}, x^{i+1},..., x^n).$ Let 
\[\KK(x):=\displaystyle\prod_{i=1}^{n}S_i(x^{-i}),\]
represents the full Cartesian product of the strategy sets. Furthermore, every player $i$ has
a convex smooth utility (payoff) function  \(U_i:\HH\to\RR\) defined as
\[U_i=\displaystyle\prod_{i\in I}S_i\]
Then it is well known that GNEP consists in finding a vector \(x^*=(x^{*1},...,x^{*n})\in\KK(x^*)\) such that 
\begin{equation}\label{gnep-qvi-vector}
\la Tx^*,x-x^*\ra\geq 0, \ \ \forall x\in\KK(x^*),
\end{equation}
where $T$ is a vector-valued function defined as 
\[T(x)=\nabla_{x^i}U_i(x^i,x^{-i}), \ \ \forall \ 1\leq i\leq n.\] 
The goal of each player is to minimize utility with respect to the only variable \(x^i\) under their control 
\begin{equation}\label{min-utility}
\min_{x^i\in\KK} U_i(x^i,x^{-i})
\end{equation}
where \(\KK\subseteq\prod_{i\in I} S_i\). A point \(x^*\in\KK\) is said to be a variational equilibrium of a game \((U,\KK)\) if 
\begin{equation}\label{variational-equilibrium}
\la\nabla U(x^*),x-x^*\ra\geq0,\ \ \forall x\in\KK.
\end{equation}
All the previous discussed formulas in this section can be reconstructed for the utility maximization problem by replacing min with max in \eqref{min-utility} such that the vector $T$ has a line search align with the anti-gradient direction on the contrary of the gradient operator direction. i.e., the maximization problem can be slightly stated in the standard form \eqref{gnep-qvi-vector}.

\begin{definition}
We say that \(x^*\in\KK\) is a Nash equilibrium of a game $(U,\KK)$ if for every player $i\in I$, 
\[U_i(x^{i*},x^{-i*})\leq U_i(x^i,x^{-i*}), \ \ \forall (x^i,x^{-i*})\in\KK.\]
Moreover, if $\KK=\displaystyle\prod_{i\in I}S_i$, we therefore have \[U_i(x^{i*},x^{-i*})=\min_{x^i\in S_i}U_i(x^i,x^{-i*}).\]
\end{definition}

\begin{theorem}[\cite{Cavazzuti}]
Let \((U,\KK)\) be a noncooperative game and $x^*$ is a Nash equilibrium then 
\begin{equation}\label{qvi-gneg}
\la\nabla U(x^*),x-x^*\ra\geq0, \quad \forall x\in x^*+I_\KK(x^*)\cap\KK.
\end{equation}
\end{theorem}
This is called a necessary optimality condition for QVI \eqref{QVI} in the form of Stampacchia-type \cite{Stampacchia} for a noncooperative game. Also, this QVI \eqref{qvi-gneg} is equivalent to the projection equation:
\[\Pi_{I_K(x^*)}(-\nabla U(x^*))=0,\]
where \(\Pi_{I_K(x^*)}\) denotes the orthogonal projection onto the internal cone \(I_K(x^*)\) such that its associated dynamical system is defined by 
\begin{equation}\label{proj-internal-cone}
\dot{x}(t)=\Pi_{I_K(x)}(-\nabla U(x)).
\end{equation}
Such projected dynamical system describes the evolution of the game from a nonstationary initial point. 
These dynamical systems are of great importance due to their geometrical explanation. Since the antigradient direction $-\nabla U(x)$ offers the players the steepest cost if $x$ stays on the boundary of the feasible set $K$, thus projecting this direction could be performed in several ways including \eqref{proj-internal-cone}. 
A different dynamical system is to exploit the projection on the tangent cone of $K$, that is  
\[\dot{x}(t)=\Pi_{T_K(x)}(-\nabla U(x)).\]
Another dynamical system is attained by performing the projection on the whole set $K$, namely,
\[\dot{x}(t)=\Pi_K(x-\alpha\nabla U(x))-x,\]
where $\alpha$ is a fixed positive constant. Steady states of the latter two projected systems coincide with variational equilibria of the game \eqref{variational-equilibrium}. However, a Nash equilibrium is not necessarily a steady state.
In this case we remark that its steady states coincide with the solutions to the quasi variational inequality \eqref{qvi-gneg}, and consequently, equal Nash equilibria of the game.

\section{Conclusion}
In this paper we proposed a new high order projected dynamical system for solving parametric quasi vairiational inequalities. This approach is novel and attribute new algorithms which can be considered interchangeably as continuous-time versions and discrete-time counterparts iterative schemes. It can be expected that the techniques described in this paper will be useful for more elaborate dynamical models, such as stochastic models, and that the connection between such dynamical models and the solutions to quasi variational inequalities will provide a deeper understanding of generalized equilibrium problems for nonconvex scenarios. Equilibrium solutions (stationary points or trajectories) of the associated dynamical system converge to the solutions of the parametric variational inequality problems by the equivalent formulation of fixed points problems and variational inequalities. The proposed implicit and explicit algorithms may be extended for a broader class of generalized quasi equilibrium problems and even beyond the convexity scope to nonconvex equilibrium variational problems.
The stability analysis of the novel dynamical system technique has been investigated. This approach usually provides qualitative behaviour of the system around the equilibrium points. One of the advantages of this approach is studying changes over time for energy-like functions without solving the differential equation analytically.
Even the applicability and leverage of the approaches into real-applications, combining third-order dynamics into quasi variational inequalities still carries various challenges due to the computational complexity when proposing composite optimization algorithms for solving such systems. 
\medskip 

Future research directions may focus on developing efficient algorithms, integrating machine learning techniques for parameter estimation, and extending the framework to stochastic environments and/or to nonmonotone manners whether on operators or in line searches for linearly convergence of algorithms. Another direction left to the future research is through exploring the dynamic constraint convex-valued set for different formats and layouts including two parameters or linear operators with mild conditions on symmetric matrices. 
\medskip

Finally, since QVI can be used to formulate the generalized Nash game in which not only each player’s
payoff function but also their strategy set depend on the other players strategies, the QVI can attract ongoing attention to game theory. Merit functions such as the gap function is a powerful tool in the equivalent optimization formulation of the variational inequality. Since gap functions possess smooth properties when the constraints are represented by nonlinear inequalities, thus constructing gap functions for the QVI is a future research aim. From the viewpoint of application, it is essentially beneficial to study generalized Nash games that use gap functions, or more generally, merit functions and this would also be worth investigation.


\medskip




\end{document}